\documentclass{article}
\usepackage{amsmath}
\usepackage{amsthm}
\usepackage{amssymb}
\usepackage{graphicx}
\usepackage{epsfig}
\def\ml{l\kern-0.035cm\char39\kern-0.03cm}

\theoremstyle{plain}
  \newtheorem{Theorem}{Theorem}
 
  \newtheorem{Lemma}{Lemma}
  \newtheorem{Corollary}{Corollary}
  \newtheorem{Conjecture}{Conjecture}
  \newtheorem{Observation}{Observation}

\theoremstyle{definition}

\begin{document}

\begin{center}
{\LARGE Total Thue colourings of graphs\\}
\end{center}

\begin{center}
\Large{Jens Schreyer}\footnote{Institute of Mathematics,
Faculty of Mathematics and Natural Sciences, Ilmenau University of Technology, Ilmenau, Germany, \quad [jens.schreyer@tu-ilmenau.de]}
and
\Large{Erika \v Skrabu\ml \'akov\'a}\footnote{Institute of Control and Informatization of Production Processes, Faculty of Mining, Ecology, Process Control and Geotechnology, Technical University of Ko\v sice, Ko\v sice, Slovakia, \quad [erika.skrabulakova@tuke.sk]}\footnote{This work was supported by the Slovak Research and Development Agency under the contract No. APVV-0482-11, by the grant VEGA 1/0130/12 and DAAD.}
\end{center}

\smallskip

\bigskip

{\abstract A total colouring of a graph is a colouring of its vertices and edges such that no two adjacent vertices or edges have the same colour and moreover, no edge coloured $c$ has its endvertex coloured $c$ too.  A weak total Thue colouring of a graph $G$ is a colouring of its vertices and edges such that the colour sequence of consecutive vertices and edges of every path of $G$ is nonrepetitive. In a total Thue colouring also the induced vertex-colouring and edge-colouring of $G$ are nonrepetitive.   The weak total Thue number $\pi_{T_w}(G)$ of a graph $G$ denotes the minimum number of colours required in every weak total Thue colouring and the minimum number of colours required in every total Thue colouring is called the total Thue number $\pi_T$. \\
Here we show some upper bounds for both parameters depending on the maximum degree or size of the graph. We also give some lower bounds and some better upper bounds for these graph parameters considering special families of graphs.}\\



\section{Introduction}

\hspace{0,4cm}
A finite sequence $R=r_1 r_2 \dots r_{2n}$ of symbols is called a {\em repetition} if $r_i=r_{n+i}$ for all $i=1,2,\dots,n$. A sequence $S$ is called {\em repetitive} if it contains a subsequence of consecutive terms that is a repetition. Otherwise $S$ is called  {\em nonrepetitive}. Nonrepetitive sequences were first studied by Axel Thue in the beginning of the last century as a part of the investigation of word structures. In his famous paper from 1906 \cite{Th06} he showed the existence of arbitrarily long nonrepetitive sequences consisting only of three different symbols. Later these sequences found widespread applications not only in mathematics, but also in informatics, data security management and elsewhere. Their first appearance in graph theory was in 1987 (see Currie \cite{Cur87}) but the investigation of nonrepetitive graph colourings began with the seminal paper of Alon et. al. from 2002 \cite{AGHR02}.\\
Let $\varphi$ be a colouring of the vertices of a graph $G$.  We say that $\varphi$ is a {\em nonrepetitive vertex-colouring} of $G$  if for any simple path on vertices $v_1$, $v_2, \dots, v_{2n}$ in $G$ the associated sequence of colours $\varphi(v_1)$ $\varphi(v_2) \dots \varphi(v_{2n})$ is not a repetition. The minimum number of colours in a nonrepetitive vertex-colouring of a graph $G$ is the {\em Thue chromatic number} $\pi(G)$. Analogously {\em nonrepetitive edge-colourings} and the {\em Thue chromatic index} $\pi'(G)$ are defined.
The original paper of Alon et. al. \cite{AGHR02} introduced both variants of colouring however it foccused on the  edge variant. Here it was proved that for arbitrary graph $G$ it holds $\pi'(G)\leq (2\cdot e\sp{16}+1)\Delta^2$, where $\Delta$ is the maximum degree of the considered graph. Unhappily the Thue chromatic index of the graph $G$ is here called the Thue number of $G$ and denoted $\pi(G)$. This was the cause of a lot of misunderstanding and misprints in several subsequent papers that gave more attention to the vertex-colouring variant for which the upper bound of the same form: $C\cdot\Delta^2$, is known. The constant $C$ was improved several times and the best known bound today is by Dujmovi\'c et. al. \cite{Dujmovic} who could show that for large graphs $C$ even tends to 1. The bound is almost best possible because it is known that there are infinitely many graphs with Thue chromatic number at least $c\cdot\Delta^2/\log~\Delta$.
There are some classes of graphs, where the Thue chromatic number is known exactly. For paths it was already shown by Thue himself \cite{Th06}, and Currie \cite{Cur02} showed that for every cycle of length $n\in\{5$, $7$, $9$, $10$, $14$, $17\}$ $\pi(C_n)=4$  and  for other lengths of cycles on at least $3$ vertices $\pi(C_n)=3$.  In \cite{Gr07} various questions concerning nonrepetitive colourings of graphs have been formulated and consequently, also a lot of their variations appeared in the scientific literature (see e.g. Bar\'at and Czap \cite{BaCz11},  Czervi\' nski and Grytczuk \cite{CzGr07}, Grytczuk et. al. \cite{GKM10}, \cite{GPZ10} or Schreyer and \v Skrabu\ml \' akov\' a \cite{JSES11}).\\
The purpose of this paper is a first look at nonrepetitive total colourings of a graph. A (proper) total colouring of a graph is a colouring of its vertices and edges, where no two adjacent vertices or edges have the same colour and moreover, no edge has the same colour as an incident vertex. We want to apply the concept of nonrepetitive colourings to total graph colourings and do it in two different ways. If a colouring $\varphi$ of the vertices and edges of a graph $G$ has the property that the colour sequence of consecutive vertex- and edge-colours of every path in $G$ is nonrepetitive, we call $\varphi$ a {\em weak total Thue colouring}. If moreover, both the induced vertex- and edge-colourings are nonrepetitive as well, we call $\varphi$ a {\em (strong) total Thue colouring} of $G$. The minimum number of colours appearing in such a colouring is called  {\em  weak total Thue chromatic number } $\pi_{T_w}(G)$ for the first case and  {\em total Thue chromatic number} $\pi_T(G)$ for the latter case. Note that while every total Thue colouring is a proper total colouring, this does not need to be the case for weak total Thue colourings, because two adjacent vertices or edges may have the same colour.\\
In this paper we show that the total Thue chromatic number is lesser than $15\cdot\Delta\sp{2}$, where $\Delta\geq 3$ is the maximum degree of the graph. The bound is extended $18\Delta^2$ for the list version of the problem.  For the weak total Thue chromatic number we show $\pi_{T_w}\leq|E(G)|-|V(G)|+5$, that for planar graphs with $k$ faces gives $\pi_{T_w}\leq 3+k$. We also give some upper and lower bounds for these parameters considering special classes of graphs.

\section{Basic observations and preliminary lemmas}

\hspace{0,4cm} We will use the following Lemma, that  was proved in \cite{HJSS09}.
For a sequence of symbols $S=a_1$ $a_2\dots a_n$ with $a_i\in\mathbb{A}$ for some set $\mathbb A$, for all $1\leq k\leq l\leq n$, the block $a_k$ $a_{k+1}\dots a_{l}$ is denoted by $S_{k,l}$  here.

\begin{Lemma} {(\cite{HJSS09})}\label{1}
Let $A=a_1 a_2 \dots a_m$ be a nonrepetitive sequence with $a_i\in\mathbb{A}$ for all \    $i=1,2,\dots,m$. Let $B^i=b^i_1 b^i_2 \dots b^i_{m_i}$; $0\leq i\leq r+1$, be nonrepetitive sequences with $b\sp{i}_j\in\mathbb{B}$ for all \  $i=0,1,\dots,r+1$ and $j=1,2,\dots,m_i$. If $\mathbb{A}\cap\mathbb{B}=\emptyset$, then $S=B^0$ $A_{1,n_1}$ $B^1$ $A_{n_1+1,n_2}\dots B^r$ $A_{n_r+1,m}$ $B^{r+1}$ is a nonrepetitive sequence.
\end{Lemma}

A sequence of length $k$ consisting of $k$ different symbols is called a {\em  rainbow} sequence. A rainbow sequence  is trivially nonrepetitive and  if each sequence $B\sp{i}$; $0\leq i\leq r+1$, from Lemma \ref{1} consists of only one element, then it is also trivially nonrepetitive:

\begin{Corollary} \label{2}
Let $A=a_1 a_2 \dots a_m$ be a rainbow sequence with $a_i\in\mathbb{A}$ for all \  $i=1,2,\dots,m$. For $i=0,1,\dots,r+1$ let $B^i\notin\mathbb{A}$. Then $S=B^0$ $A_{1,n_1}$ $B^1$ $A_{n_1+1,n_2}\dots B^r$ $A_{n_r+1,m}$ $B^{r+1}$  is a nonrepetitive sequence.
\end{Corollary}

Moreover, it is easy to see that every nonrepetitive vertex-colouring of $G$ with $\pi(G)$ colours together with one additional colour (not used for the colouring of the vertices of $G$) used to colour all the edges of $G$ gives a weak total Thue colouring of $G$ according to Corollary~\ref{2}. A similar argument holds for nonrepetitive edge-colourings of a graph $G$. Hence, $\pi_{T_w}\leq \pi(G)+1$ and $\pi_{T_w}\leq \pi'(G)+1$. Therefore, we have:

\begin{Observation} \label{obe}
 $\pi_{T_w}\leq \min\{\pi(G), \pi'(G)\}+1$.
\end{Observation}

Moreover, the upper bound in Observation \ref{obe} is tight. To see this it is enough to consider an arbitrary star $K_{1,n}$ and its nonrepetitive colouring. Obviously $\pi( K_{1,n})=2$,  $\pi'( K_{1,n})=n$, for $n\geq 2$ $\pi_{T_w}( K_{1,n})= \min\{\pi( K_{1,n}), \pi'(K_{1,n})\}+1=3$ (as there exists no nonrepetitive sequence of length $4$ over a $2$ symbol alphabet) and $\pi_{T_w}( K_{1,1})=2= \min\{\pi( K_{1,1}), \pi'( K_{1,1})\}+1$. \\
As every total Thue colouring is also a weak total Thue colouring and in a total Thue colouring both, the edge-colouring and the vertex-colouring of the graph have to be nonrepetitive, we have:

\begin{Observation}\label{one}
$\pi_{T_w}\leq \pi_T$,
$\pi(G)\leq\pi_T$  and
$\pi'(G)\leq\pi_T$.
\end{Observation}

On the other hand, if we colour all vertices of the graph $G$ nonrepetitively with $\pi(G)$ colours and we use another $\pi'(G)$ colours to colour all edges of $G$ nonrepetitively, by Lemma~\ref{1} we obtain a total Thue colouring of $G$. Hence, the following is true:

\begin{Observation}\label{3}
$\max\{\pi'(G), \pi(G)\}\leq\pi_T\leq\pi(G)+\pi'(G)$.
\end{Observation}

For $n\geq 2$ we have $\pi_{T_w}( K_{1,n})=3$, but $\max\{\pi'(G), \pi(G)\}=n\leq\pi_{T}(K_{1,n})$, which gives the next observation.

\begin{Observation}
The difference between $\pi_{T_w}$ and $\pi_T$  can be arbitrarily large.
\end{Observation}

Weak total Thue colourings are closely related to nonrepetitive vertex-colourings of subdivided graphs as there is an easy 1-1 correspondence between weak total Thue colourings of a graph $G$ and nonrepetitive vertex-colourings of the graph $\tilde G$ which is obtained from $G$ by subdividing every edge. Hence, we have:

\begin{Observation}\label{sd}
If $\tilde G$ is the graph obtained from $G$ by subdividing every edge then: $\pi_{T_w}(G)=\pi(\tilde G)$
\end{Observation}

Together with Currie's result on nonrepetitive colourings of cycles this immediately implies the following:

\begin{Corollary}
For the cycle $C_n$ on $n$ vertices it holds $\pi_{T_w}(C_n)=4$ if $n=5$ or $7$. Otherwise $\pi_{T_w}(C_n)=3$.
\end{Corollary}

\begin{Observation}
There exists a graph where $\pi_T=\pi_{T_w}=\pi(G)=\pi'(G)$ .
\end{Observation}

From the above considerations it already follows that $\pi(C_{5})=\pi'(C_{5})=\pi_{T_w}(C_5)=4$. Figure~ $\ref{C5}$ shows, that 4 colours are also enough for a total Thue colouring of the cycle $C_5$.

\begin{figure}[h]
\begin{center}
\includegraphics[width=3.3cm]{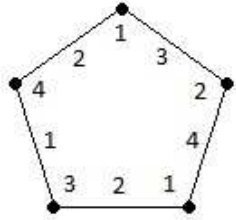}
\end{center}
\caption{Strong total Thue colouring of $C_5$}\label{C5}
\end{figure}

\section{Some general results}


\begin{Theorem}\label{subdiv}
Let $G=G(V,E)$ be a graph; $|V(G)|=n$, $|E(G)|=m$. Then $\pi_{T_w}\leq m - n+5$.
\end{Theorem}

\begin{proof}
Consider a spanning tree $T$ of $G$. If every edge of $T$ is subdivided by one vertex the tree remains a tree and therefore has a Thue chromatic number equal to 4 (see \cite{BGKNP07}). Hence (by Observation~\ref{sd}), there is a weak total Thue colouring of $T$ using 4 colours. If the remaining $m-n+1$ edges of $G$ are coloured by different colours, all paths obviously remain nonrepetitive. Hence, $\pi_{T_w}(G)\le m-n+5$.
\end{proof}

\begin{Corollary}\label{planar}
Let $G=G(V,E,F)$ be a plane graph; $|F(G)|=k$. Then $\pi_{T_w}\leq 3 + k$.
\end{Corollary}

\begin{proof}
By the Euler formula for every plane graph $G=G(V,E,F)$ with $|V(G)|=n$, $|E(G)|=m$,  $|F(G)|=k$  it holds $n+k=m+2$. Then from Theorem~\ref{subdiv} it follows that $\pi_{T_w}\leq 3 + k$.
\end{proof}

\begin{Theorem}\label{outerplanar}
Let $G$ be an outerplanar graph on $n$ vertices.
Then for $n\in\{1,2,3\}$:  $\pi_{T_w}=n$, \  for $n=4$:   $\pi_{T_w}=n-1$ \
and for $n>4$: $\pi_{T_w}\leq \min\{13, n+1\}$.
\end{Theorem}

\begin{proof}
Let $G$ be an outerplanar graph on $n\in\{1,2,3\}$ vertices.  Obviously $\pi(K_1)=1$, $\pi(P_2)=2$ and by Currie's theorem \cite{Cur02} $\pi(C_6)=3$.  One can obtain all these graphs by subdividing each edge of an outerplanar graph on $1$, $2$ or $3$ vertices respectively, by one vertex.  According to Observation~\ref{sd} this gives $\pi_{T_w}=n$ for $n\in\{1,2,3\}$.
\begin{figure}[h]
\begin{center}
\includegraphics[width=3.1cm]{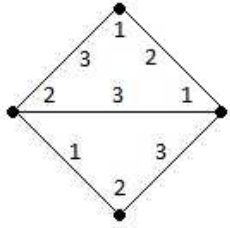}
\end{center}
\caption{A weak total Thue colouring of a diamond graph} \label{diamond}
\end{figure}
Every outerplanar graph $G$ on $4$ vertices is a subgraph of a diamond graph - depicted on Figure~\ref{diamond} together with its weak total Thue colouring using $3$ colours. Hence $\pi_{T_w}\leq 3$. On the other hand every spanning tree of the diamond graph contains as a subgraph a path on $2$ edges. Subdividing each edge of it  by one vertex one can obtain a path $P_4$; $\pi(P_4)=3$, and from  Observation~\ref{sd} it follows that the weak total Thue number of an outerplanar graph $G$ on $4$ vertices is at least $3$. Hence $\pi_{T_w}=n-1$. \\
The general result $\pi_{T_w}(G)\le\min\{13,n+1\}$ follows from Lemma~\ref{1}, Observation~ \ref{obe} and the fact, that every outerplanar graph admits a nonrepetitive vertex-colouring with 12 colours (Bar\'at and Varj\'u \cite{BaVa07}).
\end{proof}

\begin{Theorem}\label{edgetree}
Let $G$ be a graph containing $b$ bridges $e_1,e_2,\dots,e_b$ the removal of which separates $G$ into $b+1$ $2$-edge-connected components $B_1, B_2,\dots,B_{b+1}$.
Then $\pi_{T_w}\leq 4\Delta(T)-4+\max\limits_i\{\pi_{T_w}(B_i)\}$
\end{Theorem}

\begin{proof}
\textbf{The colouring algorithm:} \\
\textbf{1.} Denote $T$ the tree obtained from $G$ by contracting the components $B_1$,...,$B_{b+1}$  into single vertices. \\
\textbf{2.} Colour the edges of $T$ (i.e. $e_1,...,e_b$) nonrepetitively. According to the Theorem proved in \cite{AGHR02}  at most $4\Delta(T)-4$  colours are needed. \\
\textbf{3.} To obtain the weak total Thue colouring of $G$ find a weak total Thue colouring of each component $B_i$ (with colours different from the  colours used to colour the edges of $T$); As there is no repetitive path in each block $B_i$, by Lemma~\ref{1} the colouring obtained by the algorithm described above gives a weak total Thue colouring of $G$ with the claimed number of colours.
\end{proof}

From Thue's theorem and Observation~\ref{sd} it is obvious that the weak total Thue chromatic number of paths on at least 3 vertices is 3. From Lemma~\ref{1} it can be seen, that a total Thue colouring of every path with 6 colours can be constructed by combining a nonrepetitive vertex-colouring on 3 colours and a nonrepetitive edge-colouring on another 3 colours. The following Theorem improves this bound.

\begin{Theorem}\label{path}
For every path $P$ on at least 4 vertices it holds $4\le \pi_T(P)\le 5$.
\end{Theorem}

\begin{proof}
To see the lower bound, assume there is a total Thue colouring of $P$ using only three colours 1,2,3. Consider the colour sequence of the first 3 vertices and edges. Because such a colouring is also a proper total colouring, every colour in the sequence must differ from the two preceding  colours. Then up to renaming of the colours the sequence has to be 123123 which is repetitive, a contradiction.\\
For the upper bound we construct a colouring using 5 colours.
Let $P=v_0,e_1,v_1,...,e_n,v_n$ be a path of length $n\ge 4$. W.l.o.g. we can suppose that $n$ is divisible by $4$, as every other path is subgraph of such a path.\\

\textbf{The colouring algorithm:} \\
\textbf{1.} For all $m$ divisible by 4 colour the vertex $v_m$ with colour 4.\\
\textbf{2.} Let $s_0,s_1,s_2,...,s_{n-1}$ be a nonrepetitive sequence on $\{1,2,3\}$. Then for each $0\le i<n$ colour the vertex $v_{4i+s_i}$ with colour 5. That means between any two vertices of colour 4 there is a vertex of colour 5 and the sequence of distances between the colour 5 vertices to the preceding colour 4 vertices is nonrepetitive.\\
\textbf{3.} Whenever there are two uncoloured vertices between a vertex of colour 4 and 5, colour the edge connecting them with colour 5.\\
\textbf{4.} Colour all uncoloured edges using a nonrepetitive sequence on $\{1,2,3\}$.\\
\textbf{5.} For every vertex that is adjacent to a vertex of colour 4 and 5 use a colour from $\{1,2,3\}$ different from the colours of the neighbouring edges.\\
\textbf{6.} For two adjacent uncoloured vertices in between two vertices of colour 4  consider the edge-colours that appear between the colour 4 vertices. If one colour from $\{1,2,3\}$ is missing, colour the middle vertex with this colour and the other one with a colour from $\{1,2,3\}$ that is different from this one and the colour of the neighboring edge of colour 1, 2 or 3. If all colours of edges appear, the sequence of vertex- and edge-colours between the two vertices of colour 4 has to be $a5bx5yc$, where $a,b,c$ are different edge-colours from $\{1,2,3\}$ and $x$ and $y$ are the vertex-colours to be chosen. Choose $x=c$ and $y=a$. \\

From Lemma~\ref{1} it immediately follows, that there is no repetitive sequence of edge-colours.\\
Assume there is a repetitive sequence of vertex-colours and it contains at least one vertex of colour 4. Then it contains an even number of vertices of colour 4 and exactly as many vertices of colour 5. If the first vertex of colour 4 or 5 has colour 4, then the sequence of distances from the vertices of colour 5 to the preceding vertex of colour 4 is repetitive, a contradiction. In case the first vertex of colour 4 or 5 has colour 5, then the sequence of distances of the vertices of colour 5 to the next vertex of colour 4 must be repetitive. This is a contradiction because if the sequence $\{s_i\}_{i=0}^{n-1}$ is nonrepetitive, the sequence $\{4-s_i\}_{i=0}^{n-1}$ is nonrepetitive, too. Hence, no repetitive sequence of vertex-colours can contain a vertex of colour 4. That means, a repetition of vertex-colours can contain only two elements and adjacent vertices are coloured differently by construction.\\
Now assume that there is a colour sequence of consecutive vertex- and edge-colours. If it contains a colour 4, then this is a vertex-colour that can only be repeated by another vertex-colour. If this is the case, vertex-colours are repeated by vertex-colours and edge-colours by edge-colours. Hence, the subsequences of vertex- and edge-colours must be repetitive themselves. This is not possible as every sequence of consecutive edge-colours is nonrepetitive. Therefore, a repetition can contain at most 3 vertex-colours. All vertices are coloured differently  from their edge neighbours, and the repetition cannot consist of two vertex- and two edge-colours because otherwise two adjacent edges would have the same colour. The only remaining possibility is a sequence of three consecutive  edge- and three consecutive vertex-colours, non of which is colour 4. Now it is easy to see, that these repetitions are excluded by construction steps 5 and 6. Consequently, no repetition of any kind occurs and the constructed colouring is a total Thue colouring.
\end{proof}

In general we conjecture the following:

\begin{Conjecture}
There is an integer  $n$ such that for every path $P$ on at least $n$ vertices $\pi_T(P)=5.$
\end{Conjecture}

An immediate consequence for cycles is the following:

\begin{Corollary}
For every cycle $C$ on at least 4 vertices it holds $4\le\pi_T(C)\le 6.$
\end{Corollary}

This can be achieved by choosing one edge of a unique colour and colour the remaining path as before. But in many cases, at least if the number of vertices is large enough and divisible by 4, the colouring strategy from the previous theorem can be applied directly to generate a colouring with 5 colours. \\

The following theorem gives the exact values of the total Thue numbers of stars.

\begin{Theorem}
Let $S_n=K_{1,n}$ be a star on $n+1\geq 4$ vertices. Then $\pi_T(S_n)=n+1$.
\end{Theorem}

\begin{proof}
All the edges of $S_n$ are adjacent to each other, and therefore, they have to be coloured with different colours in every strong total Thue colouring $\varphi$ of $S_n$. They are incident with the central vertex $v$ of the star as well, therefore, $v$ has to be coloured with a new colour under $\varphi$. Hence, $\pi_T(S_n)\geq n+1$.
In order to obtain a strong total Thue colouring of the star $S_n$ colour the uncoloured vertices $w_1$,
$w_2,\dots w_n$ as follows: let $\varphi(w_n)=\varphi(w_1 v)$ and for $i=1,2,\dots n-1$ let $\varphi(w_i)=\varphi(w_{i+1} v)$. \\
All the vertices of $S_n$ are coloured with different colours and all the paths on vertices and edges of $S_n$ are coloured with a colour sequence of the form $abcdb$ or $abcde$, therefore, $\varphi$ is a strong total Thue colouring using $n+1$ colours.
\end{proof}

\section{Bounds depending on the maximum degree}

\begin{Theorem}\label{hlavna}
Let $G$ be a graph with maximum degree $\Delta\geq 3$. Then $\pi_T < 15\Delta^2$.
\end{Theorem}
\begin{proof}
Let $G$ be a graph of maximum degree $\Delta\geq 3$. Dujmovi\'c et. al. \cite{Dujmovic} proved that $\pi(G)<3\Delta^2$. By considering the line graph $H$ of $G$ (which has a maximum degree of less than $2\Delta$ this implies for the edge version of the problem $\pi'(G)\le \pi(H)<12\Delta^2$. Together with Observation \ref{3} this implies $\pi_T(G)<15 \Delta^2.$
\end{proof}

Note, that the upper bound $3\Delta^2$ on the Thue chromatic number used in the proof can be improved for larger $\Delta$. The actual bound given in \cite{Dujmovic} is $\Delta^2+o(\Delta^2)$, which with the same arguments as above implies $\pi_T(G)<5\Delta^2+o(\Delta^2)$.\\

Our last result is an extension of the above result to list colourings. The graph $G$ is {\em  nonrepetitively total $l$-choosable} if for every list assignment
$L:(V\cup E) \rightarrow 2\sp{\mathbb{N}}$ with minimum list size at least $l$ there exists a total Thue colouring $\varphi_L$ with colours from the associated lists. The {\em total Thue choice number} of a graph $G$ is the minimum number $l$, such that $G$ is nonrepetitively total $l$-choosable. (One can similarly define also the weak total Thue choice number of a graph). A bound on this parameter cannot be proved by considering vertex- and edge-colourings separately because it cannot be guaranteed, that the used colour sets of both colourings will be distinct.\\
We will use a probabilistic approach to prove our result.  In probability theory, if a large number of events are all independent of one another and each has probability less than $1$, then there is a positive probability that none of the events will occur. The Lov\'asz Local Lemma (see Erd\H{o}s and Lov\'asz \cite{ErLo75}) allows one to relax the independence condition slightly: As long as the events are "mostly" independent from one another and aren't individually too likely, then there is a positive probability that none of them occurs.
There are several different versions of the lemma - see Alon and Spencer \cite{AlSp00}. We will use the asymmetric one formulated below:

\begin{Lemma}\label{ALL}
Let $\mathcal{A}=\{A_1,A_2,\dots,A_n\}$ be a finite set of events in the probability space $\Omega$. For $A\in\cal{A}$ let $\Gamma(A)$ denote a subset of $\cal{A}$ such that $A$ is independent from the collection of events
$\mathcal{A}\setminus (\{A\}\cup \Gamma(A))$. If there exists an assignment of reals $x:\mathcal{A}\rightarrow (0;1)$ to the events such that $\forall A \in \mathcal{A}: P(A) \leq x(A) \prod_{B\in \Gamma(A)} (1-x(B))$ then the probability of avoiding all events in $\mathcal{A}$ is positive, in particular
$P(\overline{A_1}, \overline{A_2}, \dots, \overline{A_n})\geq \prod_{A\in\mathcal{A}} (1-x(A)).$
\end{Lemma}

\begin{Theorem}\label{pLLL}
For every graph with maximum degree at most $\Delta$; $\Delta\geq 3$, the total Thue choice number is at most $17.9856\Delta^2$.
\end{Theorem}

\begin{proof}
Let $G$ be a graph with maximum degree at most $\Delta$, where every vertex and edge is endowed with a list of at least $17.9856\Delta^2$ colours.
To fulfil the conditions of Lov\'asz Local Lemma we suppose that the colour of each vertex and edge is chosen randomly, independently and equiprobably out of its list. We consider the following types of bad events that may happen when this procedure is applied:

\begin{itemize}
\item For every path $P_t$ on $2t$ vertices let $A_{P_t}$ denote the event that the colour sequence of the first $t$ vertices is the same as the colour sequence of the second $t$ vertices. For the probability of the event we have $Pr(A_{P_t})\le \left(\frac{1}{17.9856\Delta^2}\right)^t$. We assign the number $x_{P_t}=\frac{1}{1+a^t}$ to the event $A_{P_t}$, where $a=7.5\Delta^2$.

\item For every path $Q_t$ on $2t$ edges let $B_{Q_t}$ denote the event that the colour sequence of the first $t$ edges is the same as the colour sequence of the second $t$ edges. For the probability of the event we have $Pr(B_{Q_t})\le \left(\frac{1}{17.9856\Delta^2}\right)^t$.We assign the number $y_{P_t}=\frac{1}{1+b^t}$ to the event $B_{Q_t}$, where $b=7.5\Delta^2$.

\item For every path $R_t=(v_1,e_1,v_2,e_2,...,v_t,e_t)$ on $t$ vertices together with the internal $t-1$ edges and one edge incident with the final vertex $v_t$ let $C_{R_t}$ denote the event that the colour sequence of the first $t$ elements (vertices and edges) of $R_t$ is the same as the colour sequence of the second half. For the probability of the event we have $Pr(C_{R_t})\le \left(\frac{1}{17.9856\Delta^2}\right)^t$.We assign the number $z_{P_t}=\frac{1}{1+c^t}$ to the event $C_{R_t}$, where $c=10\Delta^2$.
\end{itemize}

For an arbitrary event $A_{P_t}$ let $A_s$ denote the set of paths on $2s$ vertices sharing at least one vertex with $P_t$ and $C_s$ the set of paths $R_s$ on $s$ vertices and $s$ edges sharing at least one vertex with $P_t$. It it easy to see,
that $|A_s|\le 2ts\Delta^{2s}$ and $|C_s|\le 2ts\Delta^s\le \frac{2}{3}ts\Delta^{2s}$, as $\Delta\geq 3$.
We will show that
\begin{equation}\label{eins}
Pr(A_{P_t})\le x_{P_t}\cdot\prod\limits_{s=1}^\infty\prod\limits_{P_s\in A_s\setminus\{P_t\}}(1-x_{P_s})\prod\limits_{R_s\in C_s}(1-z_{R_t})
\end{equation}
Consider the right hand side $RHS_1$ of inequality (\ref{eins}):
\begin{eqnarray*}
RHS_1&=&x_{P_t}\cdot\prod\limits_{s=1}^\infty\prod\limits_{P_s\in A_s\setminus\{P_t\}}(1-x_{P_s})\prod\limits_{R_s\in C_s}(1-z_{R_t})\\
&=&\frac{1}{1+a^t}\cdot \frac{1+a^t}{a^t}\cdot\prod\limits_{s=1}^\infty\prod\limits_{P_s\in A_s}\left(1-\frac{1}{1+a^s}\right)\prod\limits_{R_s\in C_s}\left(1-\frac{1}{1+c^s}\right)\\
&\ge& \frac{1}{a^t}\prod\limits_{s=1}^\infty\left(\frac{a^s}{1+a^s}\right)^{2ts\Delta^{2s}}\left(\frac{c^s}{1+c^s}\right)^{\frac23ts\Delta^{2s}}\\
&\ge&\frac{1}{a^t}\prod\limits_{s=1}^\infty \left(e^{-\frac{1}{a^s}}\right)^{2ts\Delta^{2s}}\left(e^{-\frac{1}{c^s}}\right)^{\frac23ts\Delta^{2s}}, \end{eqnarray*}
since for all positive $x$ it holds that $\frac{x}{1+x}>e^{-\frac 1x}$.  Moreover,
\begin{eqnarray*}
\frac{1}{a^t}\prod\limits_{s=1}^\infty \left(e^{-\frac{1}{a^s}}\right)^{2ts\Delta^{2s}}\left(e^{-\frac{1}{c^s}}\right)^{\frac23ts\Delta^{2s}}&=&\frac{1}{a^t}\cdot\left(e^{-2\sum\limits_{s=1}^\infty s\cdot\left(\frac{\Delta^2}{a}\right)^s-\frac23\sum\limits_{s=1}^\infty s\cdot\left(\frac{\Delta^2}{c}\right)^s}\right)^t= \\
=\frac{1}{(7.5\Delta^2)^t}\cdot\left(e^{-2\sum\limits_{s=1}^\infty s\cdot\left(\frac{1}{7.5}\right)^s-\frac23\sum\limits_{s=1}^\infty s\cdot\left(\frac{1}{10}\right)^s}\right)^t
&=&\frac{1}{(7.5\Delta^2)^t}\cdot\left(e\sp{\frac{-4490}{10266.75}}\right)^t ,
\end{eqnarray*}
as $\sum\limits_{s=1}\sp{\infty} {s\cdot x\sp{s}}=\frac{x}{(x-1)\sp{2}}$.
Hence, $RHS_1\ge\left(\frac{0.6457}{7.5\Delta^2}\right)^t > \left(\frac{1}{17.9856\Delta^2}\right)^t$, what  proves inequality (\ref{eins}).\\

For an arbitrary event $B_{Q_t}$ let $B_s$ denote the set of paths on $2s$ edges sharing at least one edge with $Q_t$ and $C_s$ the set of paths $R_s$ on $s$ vertices and $s$ edges sharing at least one edge with $Q_t$. It is easy to see that $|B_s|\le 4ts\Delta^{2s}$ and $|C_s|\le 4ts\Delta^s\le \frac43ts\Delta^{2s}$.
Similarly as in the previous case we will show that
\begin{equation}\label{zwei}
Pr(B_{Q_t})\le y_{Q_t}\cdot\prod\limits_{s=1}^\infty\prod\limits_{Q_s\in B_s\setminus\{Q_t\}}(1-y_{Q_s})\prod\limits_{R_s\in C_s}(1-z_{R_t}).
\end{equation}
Consider the right hand side $RHS_2$ of inequality (\ref{zwei}):
\begin{eqnarray*}
RHS_2&=&y_{Q_t}\cdot\prod\limits_{s=1}^\infty\prod\limits_{Q_s\in B_s\setminus\{Q_t\}}(1-y_{Q_s})\prod\limits_{R_s\in C_s}(1-z_{R_t})\\
&=&\frac{1}{1+b^t}\cdot \frac{1+b^t}{b^t}\cdot\prod\limits_{s=1}^\infty\prod\limits_{Q_s\in B_s}\left(1-\frac{1}{1+b^s}\right)\prod\limits_{R_s\in C_s}\left(1-\frac{1}{1+c^s}\right)\\
&\ge& \frac{1}{b^t}\prod\limits_{s=1}^\infty\left(\frac{b^s}{1+b^s}\right)^{4ts\Delta^{2s}}\left(\frac{c^s}{1+c^s}\right)^{\frac43ts\Delta^{2s}}\\
&\ge&\frac{1}{b^t}\prod\limits_{s=1}^\infty \left(e^{-\frac{1}{b^s}}\right)^{4ts\Delta^{2s}}\left(e^{-\frac{1}{c^s}}\right)^{\frac43ts\Delta^{2s}}
=\frac{1}{b^t}\cdot\left(e^{-4\sum\limits_{s=1}^\infty s\cdot\left(\frac{\Delta^2}{b}\right)^s-\frac43\sum\limits_{s=1}^\infty s\cdot\left(\frac{\Delta^2}{c}\right)^s}\right)^t\\
&=&\frac{1}{(7.5\Delta^2)^t}\cdot\left(e^{-4\sum\limits_{s=1}^\infty s\cdot\left(\frac{1}{7.5}\right)^s-\frac43\sum\limits_{s=1}^\infty s\cdot\left(\frac{1}{10}\right)^s}\right)^t
=\frac{1}{(7.5\Delta^2)^t}\cdot\left(e\sp{\frac{-8980}{10266.75}}\right)^t
\end{eqnarray*}
Hence, $RHS_2\ge\left(\frac{ 0.4170003}{7.5\Delta^2}\right)^t >\left(\frac{1}{17.9856\Delta^2}\right)^t$, what proves inequality (\ref{zwei}).\\

For an arbitrary event $C_{R_t}$ let $A_s$ denote the set of paths on $2s$ vertices sharing at least one vertex with $R_t$, $B_s$ denote the set of paths on $2s$ edges sharing at least one edge with $R_t$ and $C_s$ the set of paths $R_s$ on $s$ vertices and $s$ edges sharing at least one vertex with $R_t$. It is easy to see that $|A_s|\le ts\Delta^{2s}$, $|B_s|\le 2ts\Delta^{2s}$ and $|C_s|\le ts\Delta^s\le \frac13ts\Delta^{2s}$. We will show that
\begin{equation}\label{drei}
Pr(C_{R_t})\le z_{R_t}\cdot\prod\limits_{s=1}^\infty\prod\limits_{P_s\in A_s}(1-x_{P_s})\prod\limits_{Q_s\in B_s}(1-y_{Q_s})\prod\limits_{R_s\in C_s\setminus\{R_t\}}(1-z_{R_t})
\end{equation}
Consider the right hand side $RHS_3$ of inequality (\ref{drei}):
\begin{eqnarray*}
RHS_3&=&z_{R_t}\cdot\prod\limits_{s=1}^\infty\prod\limits_{P_s\in A_s}(1-x_{P_s})\prod\limits_{Q_s\in B_s}(1-y_{Q_s})\prod\limits_{R_s\in C_s\setminus\{R_t\}}(1-z_{R_t})\\
&=&\frac{1}{1+c^t}\cdot \frac{1+c^t}{c^t}\cdot\prod\limits_{s=1}^\infty\prod\limits_{P_s\in A_s}\left(1-\frac{1}{1+a^s}\right)\prod\limits_{Q_s\in B_s}\left(1-\frac{1}{1+b^s}\right)\prod\limits_{R_s\in C_s}\left(1-\frac{1}{1+c^s}\right)\\
&\ge& \frac{1}{c^t}\prod\limits_{s=1}^\infty\left(\frac{a^s}{1+a^s}\right)^{ts\Delta^{2s}}\left(\frac{b^s}{1+b^s}\right)^{2ts\Delta^{2s}}\left(\frac{c^s}{1+c^s}\right)^{\frac13ts\Delta^{2s}}\\
&\ge&\frac{1}{c^t}\prod\limits_{s=1}^\infty \left(e^{-\frac{1}{a^s}}\right)^{ts\Delta^{2s}}\left(e^{-\frac{1}{b^s}}\right)^{2ts\Delta^{2s}}\left(e^{-\frac{1}{c^s}}\right)^{\frac13ts\Delta^{2s}}\\
&=&\frac{1}{c^t}\cdot\left(e^{-\sum\limits_{s=1}^\infty s\cdot\left(\frac{\Delta^2}{a}\right)^s-2\sum\limits_{s=1}^\infty s\cdot\left(\frac{\Delta^2}{b}\right)^s-\frac13\sum\limits_{s=1}^\infty s\cdot\left(\frac{\Delta^2}{c}\right)^s}\right)^t\\
&=&\frac{1}{(10\Delta^2)^t}\cdot\left(e^{-3\sum\limits_{s=1}^\infty s\cdot\left(\frac{1}{7.5}\right)^s-\frac13\sum\limits_{s=1}^\infty s\cdot\left(\frac{1}{10}\right)^s}\right)^t
=\frac{1}{(10\Delta^2)\sp{t}}\cdot \left( e\sp{\frac{-8980}{10266.75}}\right)^t
\end{eqnarray*}

Hence, $RHS_3\geq\left(\frac{ 0.5634}{10\Delta^2}\right)^t > \left(\frac{1}{17.9856\Delta^2}\right)^t$, what proves inequality (\ref{drei}). \\

Since the inequalities (\ref{eins}) (\ref{zwei}) and (\ref{drei}) are valid, by the Local Lemma with positive probability none of the bad events happens. Hence, there is a total Thue colouring of the graph $G$ from lists of size $17.9856\Delta^2$.  \\
\end{proof}

\end{document}